\documentclass[11pt]{article}

\title{Data-Driven System Level Synthesis}

\usepackage{amsmath}
\usepackage{amssymb}
\usepackage{tikz}
\usepackage{pgfplots}
\usepackage{graphicx}
\usepackage{epstopdf}
\usepackage[labelfont=bf,font=footnotesize]{caption}
\usepackage[labelfont=bf,font=footnotesize]{subcaption}
\usepackage{xcolor}
\allowdisplaybreaks  
\usepackage{cite}
\usepackage{algorithmicx}
\usepackage{algpseudocode}
\usepackage{algorithm}
\usepackage{bbm}
\usepackage{makecell}
\usepackage{authblk}
\usepackage{fullpage,etoolbox}
\usepackage{hyperref}

\allowdisplaybreaks

\usepackage{amsthm}

\newtheorem{definition}{Definition} 
\newtheorem{theorem}{Theorem} 
\newtheorem{corollary}{Corollary} 
\newtheorem{remark}{Remark}
\newtheorem{lemma}{Lemma}

\usepackage{enumitem}
\setlist{nosep}

\author[1]{Anton Xue}
\author[2]{Nikolai Matni}
\affil[1]{Department of Computer and Information Science, University of Pennsylvania}
\affil[2]{Department of Electrical and Systems Engineering, University of Pennsylvania}

\date{\today}

\usepackage{amsfonts}
\usepackage{amsmath}
\usepackage{wrapfig}
\usepackage{amssymb}
\usepackage{mathdots}
\usepackage{enumitem}
\usepackage[font=small,labelfont=bf,figurename=Fig.]{caption}
\usepackage{comment}
\usepackage{float}
\usepackage{graphicx}
\usepackage{letltxmacro}
\usepackage{listings}
\usepackage{mathtools}

\let\parens\undefined
\newcommand{\parens}[1]{{\left(#1\right)}}

\let\braces\undefined
\newcommand{\braces}[1]{{\left\{#1\right\}}}

\let\norm\undefined
\newcommand{\norm}[1]{{\left\lVert#1\right\rVert}}

\let\mrm\undefined
\newcommand{\mrm}[1]{{\mathrm{#1}}}

\let\trm\undefined
\newcommand{\trm}[1]{{\textrm{#1}}}

\let\mbb\undefined
\newcommand{\mbb}[1]{{\mathbb{#1}}}

\let\mbf\undefined
\newcommand{\mbf}[1]{{\mathbf{#1}}}

\let\mcal\undefined
\newcommand{\mcal}[1]{{\mathcal{#1}}}

\let\minimize\undefined
\DeclareMathOperator*{\minimize}{minimize}

\let\rank\undefined
\DeclareMathOperator{\rank}{rank}

\newtheorem*{solution*}{Solution}

\lstdefinestyle{plainsty}{
  basicstyle=\small\ttfamily,
  language=C,
  xleftmargin=\parindent,
  aboveskip=1em,
  belowskip=1em,
  showspaces=false,
  showstringspaces=false,
  keywordstyle = {},
}

\lstnewenvironment{pcode}{\lstset{style=plainsty}}{}

\let\code\undefined
\newcommand{\code}[1]{\texttt{#1}}

\let\PHI\undefined
\newcommand{\PHI}{\mathbf{\Phi}}

\newcommand{\xt}{x_{[0,T-1]}}
\newcommand{\xtt}{\tilde x_{[0,T-1]}}
\newcommand{\xL}{x_{[0,L-1]}}

\newcommand{\ut}{u_{[0,T-1]}}
\newcommand{\uL}{u_{[0,L-1]}}

\newcommand{\wL}{w_{[-1,L-2]}}

\newcommand{\R}{\mathbb{R}}
\newcommand{\dsuper}[3]{#1^{#2,#3}}

\newcommand{\iid}{\overset{\tiny\text{iid}}{\sim{}}}

\newcommand{\DDelta}{\boldsymbol{\Delta}}
\newcommand{\OZ}{\mcal{Z}_{L}}
\newcommand{\twonorm}[1]{\lVert #1 \rVert_2}

\newcommand{\ed}{\overset{\mathrm{dist}}{=}}

\usepackage{times}

\begin{document}

 \setlength{\intextsep}{6pt}%
 \setlength{\columnsep}{6pt}%
\setlength{\belowcaptionskip}{-15pt}


\maketitle

\begin{abstract}
We establish data-driven versions of the System Level Synthesis (SLS) parameterization of achievable closed-loop system responses for a linear-time-invariant system over a finite-horizon.  Inspired by recent work in data-driven control that leverages tools from behavioral theory, we show that optimization problems over system-responses can be posed using only libraries of past system trajectories, without explicitly identifying a system model.  We first consider the idealized setting of noise free trajectories, and show an exact equivalence between traditional and data-driven SLS.  We then show that in the case of a system driven by process noise, tools from robust SLS can be used to characterize the effects of noise on closed-loop performance, and further draw on tools from matrix concentration to show that a simple trajectory averaging technique can be used to mitigate these effects.  We end with numerical experiments showing the soundness of our methods.
\end{abstract}


\numberwithin{equation}{section}

As the systems we control become increasingly complex, dynamic, heterogeneous, and difficult to model, the intricate and detailed system models needed by traditional tools from robust and optimal control, derived either from first principles or expensive and time-consuming system identification methods, will no longer be available.  Fortunately, contemporary systems are also inherently \emph{data-rich}, allowing for data-driven control algorithms to be deployed, wherein techniques rooted in machine learning take the place of the traditional system identification step.

Focusing on the control of an unknown linear system, approaches including identify-then-control \cite{dean2019sample,mania2019certainty}, policy gradient, \cite{fazel2018global,malik2019derivative,furieri2020learning}, adaptive methods based on robust control \cite{dean2018regret}, and online-learning \cite{simchowitz2020improper, hazan2020nonstochastic} have been considered.  More closely related to this paper, are methods based on data-driven methods \cite{de2019persistency,coulson2019data} rooted in behavioral theory \cite{willems1986from,willems1997introduction,willems2005note} which do not estimate the system dynamics, and instead rely directly on libraries of past system trajectories.  For a more exhaustive overview of recent developments please see \cite{matni2019self} and \cite{recht2019tour} for tutorials aimed at audiences from control-theoretic and machine learning backgrounds, respectively.

This paper is motivated by the results presented in \cite{de2019persistency,coulson2019data,coulson2020distributionally,coulson2019regularized,van2020willems,rotulo2019data}.   Broadly, these papers leverage the behavioral framework of \cite{willems1997introduction}, which allows for the achievable input/output behavior of a system to be characterized in terms of a library of past trajectories, assuming certain \emph{persistence of excitation} conditions are satisfied.  For example, in \cite{coulson2019data}, the authors show how trajectory tracking in output-feedback based model-predictive-control (MPC) \cite{garcia1989model, borrelli2017predictive} can be posed as an optimization problem over a library of past system trajectories, with follow up work establishing connections to distributionally robust programming \cite{coulson2020distributionally} and allowing for real-time implementations \cite{coulson2019regularized}.  Similarly, in \cite{rotulo2019data,de2019persistency}, it is shown that data-driven synthesis of linear quadratic regulators can be achieved through the solution of semidefinite programs without identifying an explicit system model.  We note that to the best of our knowledge, no characterization of the effects of noise in the data on the performance achieved by the data-driven controllers is provided in the aforementioned papers.

\textbf{Contributions:} In this paper, we establish \emph{data-driven} versions of the System Level Synthesis (SLS) \cite{anderson2019system} parameterization of achievable closed-loop system responses for a linear-time-invariant system over a finite-horizon, that in particular, allow for the effects of noisy data on closed-loop performance to be quantified.  SLS has been central to breakthroughs in distributed optimal control \cite{wang2019system}, robust and distributed MPC \cite{alonso2019distributed,wang2019robust}, and learning-enabled control \cite{dean2019sample, dean2018regret}: our goal in this work is to take a first step towards extending its advantages to the purely data-driven, model-free setting.  In particular, we show that optimization problems over system-responses can be posed using only libraries of past system trajectories.  We first consider the idealized setting of noise free trajectories, and show an exact equivalence.  We then show that in the case of a system driven by process noise, tools from robust SLS \cite{matni2017scalable} can be used to characterize the effects of using ``noisy'' trajectories to synthesize data-driven controllers on closed-loop performance.  We further draw on tools from matrix concentration \cite{tropp2012user} to show that a trajectory averaging can be used to mitigate these effects.

\textbf{Paper structure:} In Section \ref{sec:problem}, we formally define the problem considered in this paper.  In Section \ref{sec:dd-sls}, we derive a data-driven SLS parameterization for the case of a linear-time-invariant system with no driving noise.  In Section \ref{sec:robust-dd-sls}, we extend these results to a robust data-driven SLS parameterization to accomodate driving noise.  In Section \ref{sec:subopt}, we derive sub-optimality bounds under norm-bound assumptions on the driving noise, and in Section \ref{sec:complexity}, we characterize the sample complexity needed to achieve these norm bounds through the use of a trajectory averaging technique, thus providing end-to-end sample complexity bounds for data-driven SLS. Section \ref{sec:experiments} contains  numerical examples, and we end in Section \ref{sec:conclusion} with conclusions and discussions for future work.  

\textbf{Notation:} We use $x_{[i,j]}$ as shorthand for the signal $[x^\top(i) \ x^\top(i+1) \ \cdots \ x^\top(j) ]^\top$.  Define
\[\scriptsize
\mcal{H}_L (\sigma_{[0,T-1]})
            = \begin{bmatrix}
                \sigma(0) & \sigma(1) & \cdots & \sigma(T-L) \\
                \sigma(1) & \sigma(2) & \cdots & \sigma(T -L +1) \\
                \vdots & \vdots & \ddots & \vdots \\
                \sigma(L-1) & \sigma(L) & \cdots & \sigma(T - 1)
            \end{bmatrix}, \enskip \mathbf{R} = \begin{bmatrix}
\dsuper{R}{0}{0} & \ & \ & \ \\
\dsuper{R}{1}{1} & \dsuper{R}{1}{0} & \ & \ \\
\vdots & \ddots & \ddots & \ \\
\dsuper{R}{T-1}{T-1} & \cdots & \dsuper{R}{T-1}{1} & \dsuper{R}{T-1}{0}
\end{bmatrix}
\]
For a signal $\sigma_{[0,T-1]}$, we denote the above Hankel matrix of order $L$ by $\mcal{H}_L(\sigma_{[0,T-1]})$. When the length of the signal is clear from context, we overload notation and simply write $\mcal{H}_L(\sigma)$.   A linear, causal operator $\mathbf{R}$ defined over a horizon of $T$ has matrix representation, as shown above:
here $\dsuper{R}{i}{j} \in \mathbb{R}^{p \times q}$ is a matrix of compatible dimension. We denote the set of such linear causal operators by $\mathcal{L}_{TV}^{T, p \times q}$ and drop the superscript $T, p \times q$ when it is clear: then, an operator $\mathbf{R}\in \mathcal{L}_{TV}^{T, p \times q}$ acts on a signal $\sigma_{[0,T-1]}$ through multiplication, i.e.,  $y_{[0,T-1]} = \mathbf{R}\sigma_{[0,T-1]}$. We slightly overload notation, and use \texttt{Matlab}-like syntax to extract block matrices, rows, and columns from linear operators: $\mathbf{R}(i,j)$,  $\mathbf{R}(i,:)$, and $\mathbf{R}(:,j)$ denote the $(i,j)$-th block matrix, $i$-th block row, and $j$-th block column of $\mathbf{R}$, respectively, all indexing from $0$. 

\section{Problem Statement}
\label{sec:problem}
We consider finite-time optimal state-feedback control of the discrete time linear-time-invariant (LTI) dynamical system
\begin{equation}\label{eq:lti}
x(t+1) = Ax(t) + Bu(t) + w(t), \text{ for $t=0,1,\dots,L-1,$}
\end{equation}
where $L>0$ is the control horizon, $x\in\R^n$ is the system state, $u\in\R^m$ is the control input, and $w\in\R^n$ is the disturbance.  We assume that the pair $(A,B)$ is controllable.  In order to simplify notation going forward, we adopt the convention of embedding the initial condition of system \eqref{eq:lti} in the disturbance signal as $w(-1)=x(0)$.

When the system model $(A,B)$ is known, the above problem can be efficiently solved for many cases of interest by making suitable assumptions on the noise signal $\wL$ and control objective.  This paper focuses instead on solving an optimal control problem when the model describing system \eqref{eq:lti} is unknown, but a collection of state and input trajectories (over a longer horizon $T>L$ to be specified in Section 2) $\{\xt^{(i)},\ut^{(i)}\}_{i=1}^N$ are available.  Moreover, our goal is to solve this task \emph{without explicitly estimating} the system model.

In particular, to make the discussion concrete, we focus on finite-horizon Linear Quadratic Gaussian (LQG) control, wherein the disturbances are assumed to be independently and identically distributed as $w(t) \iid \mathcal{N}(0,\sigma^2 I)$, the control policy at time $t$ is given by a linear-time-varying function of past states, i.e., $u(t) = K_t(x_{[0, t]})$, and the cost function to be minimized is given by:
\begin{equation}\label{eq:lq-cost}
    \mathbb{E}\left[\sum_{t=0}^{L-2} x^\top(t) Q x(t) + u^\top(t) R u(t) + x^\top(L-1) Q_F x(L-1)\right].
\end{equation}
We note however that much of our analysis extends to other cost functions in a natural way.
\section{Data Driven System Level Synthesis}
\label{sec:dd-sls}
We begin by considering the simplified setting in which there is no driving noise in system \eqref{eq:lti}, i.e., $w(-1)=x(0)$ and $w(t)=0$ for all $t\geq 0$.  Our approach is to connect tools from behavioral control theory, namely Willems' Fundamental Lemma \cite{willems1997introduction, markovsky2008data}, with the \emph{System Level Synthesis} (SLS) \cite{anderson2019system} parameterization of closed loop controllers.  While such a connection offers no immediate benefits in this simplified setting of noise-free centralized control, it establishes the tools needed to tackle the general problem of interest.  

\subsection{Willems' Fundamental Lemma}
Tools from behavioral system theory \cite{willems1997introduction,willems2005note,de2019persistency} provide a natural way of characterizing the behavior of a dynamical system in terms of its input/output signals.  For our purposes, we rely on recent specializations of \emph{Willems' Fundaemental Lemma} to state-space realizations of LTI systems \cite{de2019persistency}.  Central to Willems' fundamental lemma is
\emph{persistence of excitation},
which is specified in terms of a rank condition on a Hankel matrx constructed from the of the control input signal \(u\).

\begin{definition}
    \label{def:persistence}
    Let \(\sigma : \mbb{Z} \to \mbb{R}^p\) be a signal.
    We say that its finite-horizon restriction \(\sigma_{[0, T-1]}\)
    is \emph{persistently exciting} (PE)
    of order \(L\) if the Hankel matrix \(\mcal{H}_L (\sigma_{[0, T-1]})\) has full rank.
\end{definition}

The rank condition implies that \(T \geq (p + 1) L - 1\)
is a lower-bound on the horizon \(T\). In what follows, we assume that the order $L$ and data horizon $T$ are chosen such that this bound is satisfied.

\begin{lemma}[\cite{willems2005note, de2019persistency}]
    \label{lem:fund-lemma}
    Consider the system \eqref{eq:lti} with $(A,B)$ controllable, and assume that there is no driving noise.  Let $\{\xt, \ut\}$ be the state and input signals generated by the system.  Then if $\ut$ is PE of order \(n + L\), the signals
        $\xL ^\star$ and $\uL ^\star$
        are valid trajectories \(L\)-length of system \eqref{eq:lti}
        if and only if
        \begin{align*}
            \begin{bmatrix} x_{[0, L - 1]} ^\star \\ u_{[0, L - 1]} ^\star \end{bmatrix} =
            \begin{bmatrix} \mcal{H}_L (\xt) \\ \mcal{H}_L (\ut) \end{bmatrix} g,
             \text{ for some \(g \in \mbb{R}^{T - L + 1}\)}
        \end{align*}
\end{lemma}

Lemma~\ref{lem:fund-lemma} states that if the underlying system is controllable,
and \(\rank \mcal{H}_{n + L} (u) = n + L\),
then: (1) all initial conditions and inputs are parameterizable from observed signal data;
and (2) all valid trajectories, that is to say state/input trajectory pairs $\{\xL,\uL\}$ that are consistent with the dynamics \eqref{eq:lti}, lie in the linear span of a suitable Hankel matrix constructed from the system trajectories.  Our goal is to exploit this relationship to characterize valid \emph{closed loop system responses} of the unknown system \eqref{eq:lti} by establishing a connection to the SLS parameterization.


\subsection{System Level Synthesis}
Consider an \(L\)-length trajectory from system~\eqref{eq:lti}
expressed as block matrix operations
\begin{align}
    \label{eqn:stacked-dynamics}
    \xL
        = \mcal{Z} \mcal{A} \xL
            + \mcal{Z} \mcal{B} \uL
            + \wL
\end{align}
where \(\mcal{A} = I_L \otimes A\),
and \(\mcal{B} = I_L \otimes B\),
and where \(\mcal{Z}\) is the block-downshift operator,
i.e., a matrix with identity matrices along the first block subdiagonal and zeros elsewhere.
If it is also the case that the system~\eqref{eqn:stacked-dynamics}
satisfies the linear feedback control law
\(\uL = \mcal{K} \xL\)
for a causal linear-time-varying state-feedback control policy
\(\mcal{K}\in \mathcal{L}_{TV}^{L, m\times n}\),
then rewriting \eqref{eqn:stacked-dynamics} we arrive at
\begin{align}\label{eq:responses}
    \xL
        &= (I - \mcal{Z} (\mcal{A} + \mcal{B} \mcal{K}))^{-1} \wL
        = \PHI_x \wL
        \\
    \uL \nonumber
        &= \mcal{K} (I - \mcal{Z} (\mcal{A} + \mcal{B} \mcal{K}))^{-1} \wL
        = \PHI_u \wL
\end{align}
which captures how the process noise \(w\)
maps to the state \(x\) and control \(u\).
We refer to the causal linear operators $\PHI_x\in \mathcal{L}_{TV}^{L, n \times n}$ and $\PHI_u \in \mathcal{L}_{TV}^{L, m \times n}$ as the \emph{system responses}, which characterize the closed-loop system behavior from noise to state and control input, respectively.

\begin{theorem}[Theorem 2.1, \cite{anderson2019system}]
    \label{thm:sls-thm}
    For a system~\eqref{eq:lti}
    with state-feedback control law $\mcal{K} \in \mathcal{L}_{TV}^{L, m \times n}$, i.e., \(\uL = \mcal{K} \xL\),
    the following are true
    \begin{enumerate}
        \item The affine subspace defined by
        \begin{align}
            \label{eqn:sls-constr}
            \begin{bmatrix} (I - \mcal{Z} \mcal{A}) & -\mcal{Z} \mcal{B} \end{bmatrix}
            \begin{bmatrix} \PHI_x \\ \PHI_u \end{bmatrix}
            = I, \quad \PHI_x \in \mathcal{L}_{TV}^{L, n \times n}, \PHI_u \in \mathcal{L}_{TV}^{L, m \times n}
        \end{align}
        parameterizes all possible system responses from $\wL \to (\xL, \uL)$.
    
        \item For any causal linear operators \(\PHI_x, \PHI_u\)
        satisfying \eqref{eqn:sls-constr},
        the controller \(\mcal{K} = \PHI_u \PHI_x ^{-1} \in \mathcal{L}_{TV}^{L, m \times n}\)
        achieves the desired closed-loop responses \eqref{eq:responses}.
    \end{enumerate}
    
\end{theorem}

Theorem \ref{thm:sls-thm} allows for the problem of controller synthesis to be equivalently posed as a search over the affine space of system responses characterized by constraint \eqref{eqn:sls-constr} by setting $\xL = \PHI_x\wL$ and $\uL = \PHI_u\wL$.  
In particular, the LQG problem posed in Section \ref{sec:problem} can be recast as a search over system responses (see Section 2.2 of \cite{anderson2019system}) as:\footnote{We drop both the squaring of the objective function, and the scaling factor $\sigma^2$, for brevity of notation going forward, as neither affect the optimal solution, or the order-wise scaling of the derived bounds.}
\begin{align}
    \label{eqn:sls-lqg}
    \minimize_{\PHI_x, \PHI_u}
        \enskip
        \norm{\begin{bmatrix} \mcal{Q}^{1/2} \\ & \mcal{R}^{1/2} \end{bmatrix}
                    \begin{bmatrix} \PHI_x \\ \PHI_u \end{bmatrix}}_F
    \qquad \text{subject to \eqref{eqn:sls-constr},}
\end{align}
where \(\mcal{Q} = I_L \otimes Q\), \(\mcal{R} = I_L \otimes R\), and $\|\cdot\|_F$ is the Frobenius norm.  In the noise free setting, i.e., when the initial condition $w(-1) = x(0)$ is known, and $w(t)=0$ for $t\geq 0$, the objective function of this problem instead simplifies to 
\begin{align}
    \label{eqn:sls-lqr}
    \minimize_{\PHI_x(:,0), \PHI_u(:,0)}
        \enskip
        \norm{\begin{bmatrix} \mcal{Q}^{1/2} \\ & \mcal{R}^{1/2} \end{bmatrix}
                    \begin{bmatrix} \PHI_x(:,0) \\ \PHI_u(:,0) \end{bmatrix}x(0)}_F,
\end{align}
and similarly, because only the initial condition is nonzero, the affine constraint~\eqref{eqn:sls-constr} reduces to
\begin{align}
    \label{eqn:noiseless-sls-constr}
    \begin{bmatrix} (I - \mcal{Z} \mcal{A}) & -\mcal{Z} \mcal{B} \end{bmatrix}
    \begin{bmatrix} \PHI_x (:, 0) \\ \PHI_u (:, 0) \end{bmatrix}
    = I(:,0).
\end{align}

\subsection{A Data-Driven Formulation}

We now show how the simplified achievability constraints \eqref{eqn:noiseless-sls-constr} can be replaced by a data-driven representation through the use of Lemma \ref{lem:fund-lemma}.  Our key insight is to recognize that the $i$-th column of $\PHI_x$ and $\PHI_u$ are the impulse response of the state and control input, respectively, to the $i$-th disturbance channel, which are themselves, valid system trajectories that can be characterized using Willems' fundamental lemma.

\begin{theorem}
    \label{thm:noiseless-data-sls}
    Consider the system \eqref{eq:lti} with $(A,B)$ controllable, and assume that there is no driving noise.  Suppose that a state/input signal pair $\{\xt,\ut\}$ is collected, and assume that
    $\ut$ is PE of order at least $n + L$. We then have that the set of feasible solutions to constraint \eqref{eqn:noiseless-sls-constr} defined over a time horizon $t=0,1,\dots,L-1$ can be equivalently characterized as:
    \begin{equation}\label{eq:dd-sls}
        \begin{bmatrix}
                \mcal{H}_L (x) \\ \mcal{H}_L (u)
            \end{bmatrix}
            G, \enskip \text{ for all $G \in \Gamma (x):= \{G : \mcal{H}_1 (x) G = I\}$.}
    \end{equation}
\end{theorem}
\begin{proof}
Our goal is to prove the following relationship
\begin{align*}
        \bigg\{\begin{bmatrix} \PHI_x (:, 0) \\ \PHI_u (:, 0) \end{bmatrix}
            \text{ satisfying \eqref{eqn:noiseless-sls-constr}}
        \bigg\}
        =
        \bigg\{
            \begin{bmatrix}
                \mcal{H}_L (x) \\ \mcal{H}_L (u)
            \end{bmatrix}
            G
            : G \in \Gamma (x)
        \bigg\}.
    \end{align*}
To alleviate notational burden, denote the left and right-hand sets as LHS and RHS respectively.

(\(\subseteq\))
Consider some \(\PHI_x(:,0), \PHI_u(:,0) \in \mrm{LHS}\).  By Theorem \ref{thm:sls-thm}, we then have that for any initial condition $x(0)$, that $\{\xL,\uL\} = \{\PHI_x(:,0)x(0),\PHI_u(:,0)x(0)\}$ is a valid PE system trajectory.  By Lemma~\ref{lem:fund-lemma}, we then have that for any $x(0)$, there exists a $g \in \mbb{R}^{T - L + 1}$ such that
\begin{equation}
\begin{bmatrix} \PHI_x (:, 0) \\ \PHI_u (:, 0) \end{bmatrix}x(0) = \begin{bmatrix} \mcal{H}_L (x) \\ \mcal{H}_L (u) \end{bmatrix} g.
\label{eq:dd-sls-helper}
\end{equation}

Let $x^{(i)}(0)=e_i$, $i=1,\dots,n$ be the standard basis element, and let $g_i$ be the corresponding vector such that \eqref{eq:dd-sls-helper} holds.  Concatenating the resulting expressions column-wise, we obtain the following expression:
\begin{align}
\begin{bmatrix} \PHI_x (:, 0) \\ \PHI_u (:, 0)\end{bmatrix}\begin{bmatrix}
            | & | & \cdots & | \\
            e_1 & e_2 & \cdots & e_n \\
            | & | & \cdots & |
            \end{bmatrix}
=
\begin{bmatrix} \PHI_x (:, 0) \\ \PHI_u (:, 0)\end{bmatrix}
=
\begin{bmatrix}
            \mcal{H}_L (x) \\ \mcal{H}_L (u)
            \end{bmatrix}
            \begin{bmatrix}
            | & | & \cdots & | \\
            g_1 & g_2 & \cdots & g_n \\
            | & | & \cdots & |
            \end{bmatrix}
        =: \begin{bmatrix}
            \mcal{H}_L (x) \\ \mcal{H}_L (u)
            \end{bmatrix}
            G
\label{eq:dd-sls-helper2}
\end{align}
Furthermore, from equation \eqref{eq:dd-sls-helper2}, we have that \(\mcal{H}_1 (x)G=\Phi_x(0,0)=I\), as $\mcal{H}_1$
is the first block row of \(\mcal{H}_L (x)\).  It therefore follows that \(G \in \Gamma(x)\), proving that \(\mrm{LHS} \subseteq \mrm{RHS}\).

(\(\supseteq\))
Consider a \(\{\mcal{H}_L (x)G, \mcal{H}_L (u)G\} \in \mrm{RHS}\).
Substituting these into constraint \eqref{eqn:noiseless-sls-constr}, we obtain
\begin{align*}
\begin{bmatrix} (I - \mcal{Z} \mcal{A}) & - \mcal{Z} \mcal{B} \end{bmatrix}
        \begin{bmatrix} \mcal{H}_L (x) \\ \mcal{H}_L (u) \end{bmatrix} G
    = \begin{bmatrix} \mcal{H}_1 (x) \\ 0 \end{bmatrix} G
\end{align*}
where the equality is a direct consequence of the stacked system dynamics \eqref{eqn:stacked-dynamics} with no driving noise.  As \(G \in \Gamma(x)\) by assumption, we have that \(\mcal{H}_1(x) G = I\), and thus $\{\mcal{H}_L (x)G, \mcal{H}_L (u)G\}$ define valid system responses, from which the desired result follows.
\end{proof}

Thus, if a state/input pair $\{\xt,\ut\}$ is generated by a PE input signal of order at least $n + L$, Theorem~\ref{thm:noiseless-data-sls}
gives conditions under which $\{\mcal{H}_L (x), \mcal{H}_L (u)\}$ can be used to parameterize achievable system responses for system \eqref{eq:lti} under no driving noise.  In particular, one can then reformulate the deterministic optimal control problem formulated in equations \eqref{eqn:sls-lqr} and \eqref{eqn:noiseless-sls-constr} as
\begin{align}
    \label{eqn:dd-sls-lqr}
    \minimize_{G\in\Gamma(x)}
        \enskip
        \norm{\begin{bmatrix} \mcal{Q}^{1/2} \\ & \mcal{R}^{1/2} \end{bmatrix}
                    \begin{bmatrix} \mcal{H}_L (x) \\ \mcal{H}_L (u) \end{bmatrix} G x(0)}_F.
\end{align}

\section{Robust-Data Driven System Level Synthesis}
\label{sec:robust-dd-sls}

We now turn our attention to the original stochastic LQG optimal control problem \eqref{eqn:sls-lqg}, where for notational convenience we set $w(-1)=x(0)=0$, and driving noise $w(t)\iid \mathcal{N}(0,\sigma^2 I)$ for $t= 0,\dots,T-2$.  To differentiate between the state of the noise free and noisy system, we will denote the state signal by $\tilde x_{[0,T-1]}$ when driving noise is present. This additional \emph{unmeasurable input} means that valid system trajectories can no longer be solely characterized in terms of the Hankel matrices $\mcal{H}_L(\tilde x)$ and $\mcal{H}_L(u)$, as the effect of the process noise, as captured by a corresponding Hankel matrix $\mcal{H}_L(w)$, must also be accounted for.  To address this challenge, we relate the state-trajectories of system \eqref{eq:lti} under driving noise to those of system \eqref{eq:lti} under no driving noise, and use this relationship to construct \emph{approximate system responses} that lie a bounded distance from the affine subspace defined in \eqref{eqn:sls-constr}.  We then leverage a robust SLS parameterization to bound the effects of this approximation error on the closed loop behavior.

\subsection{Robust System Level Synthesis}
We begin with a robust variant of Theorem \ref{thm:sls-thm} that characterizes the behavior achieved by a controller constructed from system responses lying near the affine subspace characterized by constraint \eqref{eqn:sls-constr}.

\begin{theorem}[Theorem 2.2, \cite{anderson2019system}]
    \label{thm:robust-sls}
    Let \(\DDelta\) be a strictly causal linear operator (i.e., its matrix representation is strictly block-lower-triangular),
    and suppose that \(\{\hat{\PHI}_x, \hat{\PHI}_u\}\)  satisfy
    \begin{align}
        \label{eqn:approx-sls-constr}
        \begin{bmatrix}
            (I - \mcal{Z} \mcal{A}) & -\mcal{Z} \mcal{B}
        \end{bmatrix}
        \begin{bmatrix} \hat{\PHI}_x \\ \hat{\PHI}_u \end{bmatrix}
        = I + \DDelta, \quad \hat\PHI_x \in \mathcal{L}_{TV}^{L, n \times n}, \hat\PHI_u \in \mathcal{L}_{TV}^{L, m \times n} 
    \end{align}
    Then the controller \(\hat{\mcal{K}} = \hat{\PHI}_u \hat{\PHI}_x ^{-1}\)
    achieves the system responses 
    \begin{equation}\label{eq:approx-responses}
    \begin{bmatrix} \xL \\ \uL
    \end{bmatrix} = \begin{bmatrix} \hat{\PHI}_x \\ \hat{\PHI}_u
    \end{bmatrix}(I + \DDelta)^{-1}\wL
    \end{equation}
\end{theorem}

Equation \eqref{eq:approx-responses} shows that the effect of the error term $\DDelta$ in the approximate achievability constraint \eqref{eqn:approx-sls-constr} is to map the original disturbance signal to  $\wL\to\tilde w_{[-1,L-2]}:= (I + \DDelta)^{-1}\wL$.  This makes clear that we must design the full system responses $\{\PHI_x,\PHI_u\}$, and not just their first block-columns as in the idealized setting considered in the previous section.  In particular, the new disturbance $\tilde w_{[-1,L-2]}$ will have full support even if $w_{[-1,L-2]}$ is only nonzero for $w(-1)=x(0)$.

\subsection{A Robust Data-Driven Formulation}
Thus our challenge is to construct causal linear operators
\(\{\hat{\PHI}_x, \hat{\PHI}_u\}\) from noisy data $\{\xtt,\ut\}$ that satisfy equation \eqref{eqn:approx-sls-constr} with as small a perturbation term $\DDelta$ as possible.
Our approach is to construct each block-column of the approximate system responses individually, and then suitably concatenate them to construct a feasible solution to constraint \eqref{eqn:approx-sls-constr}, allowing us to explicitly characterize the effect of the driving noise $w_{[-1,T-2]}$ on the the perturbation term $\DDelta$. We emphasize that we have access to only \(\{\xtt, \ut\}\), but it is instructive to also consider \(w_{[-1,T-2]}\) in our analysis.
To begin, as each column of
\(\mcal{H}_L (\tilde x), \mcal{H}_L (u), \mcal{H}_L (w)\)
satisfies~\eqref{eqn:stacked-dynamics},
it follows that
\begin{align}
    \label{eqn:data-only-constr}
    \begin{bmatrix} (I - \mcal{Z} \mcal{A}) & -\mcal{Z} \mcal{B} \end{bmatrix}
    \begin{bmatrix} \mcal{H}_L (\tilde x) \\ \mcal{H}_L (u) \end{bmatrix}
    =
    \begin{bmatrix} \mcal{H}_1 (\tilde x) \\ 0 \end{bmatrix}
    + \mcal{Z} \mcal{H}_L (w)
\end{align}
Then, fix a \(\hat{G} \in \Gamma (\tilde x)\) and 
let \(\hat{\PHI}_x (:, 0) = \mcal{H}_L (\tilde x) \hat{G}\)
and \(\hat{\PHI}_u (:, 0) = \mcal{H}_L (u) \hat{G}\)
as in the proof of Theorem~\ref{thm:noiseless-data-sls},
\begin{align}
    \label{eqn:one-col-constr}
    \begin{bmatrix} (I - \mcal{Z} \mcal{A}) & -\mcal{Z} \mcal{B} \end{bmatrix}
    \begin{bmatrix} \hat{\PHI}_x (:, 0) \\ \hat{\PHI}_u (:, 0) \end{bmatrix}
    =
    I(:,0)
    + \underbrace{\mcal{Z} \mcal{H}_L (w) \hat{G}}_{\DDelta(:,0)}
\end{align}
If block down-shifting is accounted for, \eqref{eqn:one-col-constr}
demonstrates the construction of a single block-column of
\(\{\hat{\PHI}_x, \hat{\PHI}_u\}\) and $\DDelta$.  The construction of the other columns is similar;
in general consider
\(\hat{G}_0, \ldots, \hat{G}_{L - 1} \in \Gamma(\tilde x)\),
where each \(\hat{G}_{k - 1}\) is used to construct the \(k\)th column of
\(\hat{\PHI}_x, \hat{\PHI}_u\).
Note that \(\mcal{Z}\) commutes with block-diagonal matrices with
identical block-diagonal entries (adjusting for dimensions): this can be seen by observing that left-multiplication by $\mcal Z$ (down-shifting) is equivalent
to right-multiplication by $\mcal Z$ (left-shifting).  Thus, we can construct down-shifted block-columns of the form \eqref{eqn:one-col-constr} as follows
\begin{align}
    \label{eqn:k-col-constr}
    \begin{bmatrix} (I - \mcal{Z} \mcal{A}) & -\mcal{Z} \mcal{B} \end{bmatrix}
    \begin{bmatrix}
        \mcal{Z}^{k - 1} \mcal{H}_L (x) \\
        \mcal{Z}^{k - 1} \mcal{H}_L (u)
    \end{bmatrix}
    \hat{G}_{k - 1}
    = \mcal{Z}^{k - 1} I(:,0)
        + \mcal{Z}^k \mcal{H}_L (w) \hat{G}_{k - 1},
\end{align}
from which full approximate system responses can be constructed, as formalized in the following.

\begin{theorem}
    \label{thm:noisy-data-sls}
    For system~\eqref{eq:lti} with $(A,B)$ controllable, and control input $\ut$ and disturbance process $w_{[-1,T-2]}]$ PE of order \(n + L\),
    the approximate system response matrices $\{\hat\PHI_x, \hat\PHI_u\}$ and perturbation term \(\DDelta\) are defined as
    \begin{align}
        \hat{\PHI}_x
            &= \OZ
                (I_L \otimes \mcal{H}_L (\tilde x)) \hat{\mcal{G}}
                \label{eqn:hats}, \enskip
        \hat{\PHI}_u
            = \OZ
                (I_L \otimes \mcal{H}_L (u)) \hat{\mcal{G}}
                , \enskip
        \DDelta
            = \OZ
                (I_L \otimes \mcal{Z} \mcal{H}_L (w))
                \hat{\mcal{G}}
    \end{align}
    and satisfy the approximate achievability constraint \eqref{eqn:approx-sls-constr}, where
    \(
       \OZ := \begin{bmatrix} I & \mcal{Z} & \cdots & \mcal{Z}^{L - 1} \end{bmatrix}\)
    and
    \(\hat{\mcal{G}}\in\mcal{L}_{TV}\) with block-diagonal elements {$\hat{G}(i, i)\in\Gamma(\tilde x)$} for $i=0,1,\dots, L-1$, and off-diagonal blocks $\mcal{H}_1(\tilde x)\hat{G}(i,j)=0$ for $i \neq j$.
\end{theorem}
\begin{proof}
First, we see that by construction, $\hat\PHI_x \in \mathcal{L}_{TV}^{L, n \times n}, \hat\PHI_u \in \mathcal{L}_{TV}^{L, m \times n}$, and $\DDelta\in \mathcal{L}_{TV}^{L, n \times n}$ is strictly block-lower-triangular, and thus it suffices to verify that they satisfy the approximate achievability constraint \eqref{eqn:approx-sls-constr}.  

Next, we observe that for $k=0,\dots, L-1$, it holds that
\begin{align*}
    (I_{nL \times nL}-\mcal{Z}_{nL \times nL}\mcal A)\mcal{Z}_{nL \times nL}^k &= \mcal Z_{nL \times nL}^k(I_{nL \times nL}-\mcal Z_{nL \times nL}\mcal A),\\
     -\mcal Z_{nL \times nL}\mcal B\mcal Z_{mL \times mL}^k &= -\mcal Z_{nL \times nL}^k\mcal Z_{nL \times  nL}\mcal B,
\end{align*}
where we use subscripts in the above to denote matrix dimensions of $I$ and $\mcal Z$.

Combining the above with equation \eqref{eqn:data-only-constr}, we then have that
    \begin{multline}\label{eq:approx-resp-helper}
        \begin{bmatrix} (I - \mcal{Z} \mcal{A}) & -\mcal{Z} \mcal{B} \end{bmatrix}
        \begin{bmatrix}
            \OZ (I_L \otimes \mcal{H}_L (\tilde x)) \\
            \OZ (I_L \otimes \mcal{H}_L (u))
        \end{bmatrix}
        =
        \OZ\left(I_L \otimes [(I-\mcal Z\mcal A), -\mcal Z\mcal B]\begin{bmatrix}\mcal H_L(\tilde x)\\ \mcal H_L(u)\end{bmatrix}\right)\\
        =
        \OZ
        \parens{I_L \otimes
            \parens{\begin{bmatrix} \mcal{H}_1 (\tilde x) \\ 0 \end{bmatrix} + \mcal{Z} \mcal{H}_L (w)}}
    \end{multline}
    
    Consider a \(\hat{\mcal{G}}\in\mcal{L}_{TV}\) constructed as described in the Theorem statement.
    Right-multiplying the left-hand-side of equation \eqref{eq:approx-resp-helper} by \(\hat{\mcal{G}}\) yields the desired left-hand-side of constraint \eqref{eqn:approx-sls-constr}.  Further, notice that from the definition of $\Gamma(x)$ and the construction of $\hat{\mcal G}$, we have that 
    \begin{align*}
        \OZ
        \parens{I_L \otimes \begin{bmatrix} \mcal{H}_1 (\tilde x) \\ 0 \end{bmatrix}} \mcal{\hat G} + \OZ (I_L \otimes \mcal{Z} \mcal{H}_L (w))\mcal{\hat G}
        = I_{nL} + \DDelta,
    \end{align*}
    concluding the proof.
\end{proof}

Theorem~\ref{thm:noisy-data-sls} thus allows us to apply the robust SLS parameterization of Theorem \ref{thm:robust-sls} to characterize the closed-loop behavior \eqref{eq:approx-responses} achieved by a controller constructed from the data-driven approximate system responses $\{\hat\PHI_x,\hat\PHI_u\}$ in terms of the perturbation term $\DDelta$, as described in equation \eqref{eqn:hats}.  In particular, if we assume that $\norm{\mcal H(w)}_2\leq \varepsilon$, but is otherwise acting adversarially, we can pose the following robust LQG problem:
\begin{align}
    \label{eqn:noisy-data-sls-lqr}
    \nonumber \minimize_{\hat{\mcal{G}}\in\mcal{L}_{TV}}\max_{\norm{\mcal{H}_L(w)}\leq \varepsilon}
        &\quad
            \norm{\begin{bmatrix} \mcal{Q}^{1/2} \\ & \mcal{R}^{1/2} \end{bmatrix}
            \begin{bmatrix}
            \OZ (I_L \otimes \mcal{H}_L (\tilde x)) \\
            \OZ (I_L \otimes \mcal{H}_L (u))
        \end{bmatrix}\hat{\mcal G}
            (I + \mcal{Z}_L (I_L \otimes \mcal{Z} \mcal{H}_L (w))
                \hat{\mcal{G}})^{-1}}_F \\
    \text{subject to}
            & \text{ $\hat{G}(i,i)\in \Gamma(\tilde x)$ for $i=0,1,\dots,L-1$,}\\& \text{ $\mcal{H}_1(x)\hat{G}(i,j)=0$ for all $i\neq j$}\nonumber.
\end{align}

Note that although we assume that the disturbance process is drawn as $w(t)\iid\mathcal{N}(0,\sigma^2 I)$, we conservatively treat the effects of the unknown Hankel matrix $\mcal H_L(w)$ on the estimated system responses as adversarial in our analysis.  This approach also allows our method to generalize naturally to other optimal control settings, such as those with $\mcal H_\infty$ and $\mcal L_1$ cost functions.

The objective function of optimization problem \eqref{eqn:noisy-data-sls-lqr} is a non-convex, but its structure allows for a transparent and data-independent quasi-convex upper-bound to be derived.  First, we observe that we can upper bound $\norm{\DDelta}_2$ as given in equation \eqref{eqn:hats}, by
\begin{equation}
    \twonorm{\DDelta} = \twonorm{\OZ
                (I_L \otimes \mcal{Z} \mcal{H}_L (w))
                \hat{\mcal{G}}}\leq \twonorm{\OZ}\twonorm{\mcal{H}_L (w)}\twonorm{\mcal{\hat G}}\leq \sqrt{L}\varepsilon \twonorm{\mcal{\hat G}},
                \label{eq:delta-bound}
\end{equation}
from which it follows immediately that if $\sqrt{L}\varepsilon\twonorm{\mcal{\hat G}}<1$, we have the following upper bound 
\[\twonorm{(I+\DDelta)^{-1}} \leq \frac{1}{1-\sqrt{L}\varepsilon\twonorm{\mcal{\hat G}}}.\]  

This observation allows us to follow a similar argument as in \cite{dean2019sample} to derive the following quasi-convex upper bound to problem \eqref{eqn:noisy-data-sls-lqr}:
\begin{align}
    \label{eqn:robust-data-sls-lqr}
    \minimize_{\gamma \in [0, 1), \, \hat{\mcal{G}}}
        &\quad
            \frac{1}{1 - \gamma}
            \norm{\begin{bmatrix} \mcal{Q}^{1/2} \\ & \mcal{R}^{1/2} \end{bmatrix}
            \begin{bmatrix}
            \OZ (I_L \otimes \mcal{H}_L (\tilde x)) \\
            \OZ (I_L \otimes \mcal{H}_L (u))
        \end{bmatrix}\hat{\mcal G}}_F \\
    \text{subject to}
        &\quad \twonorm{\mcal{\hat G}}\leq \frac{\gamma}{\sqrt{L}\varepsilon},\text{ $\hat{G}(i,i)\in \Gamma(\tilde x)$ for $i=0,1,\dots,L-1$, }\nonumber\\&\quad\text{$\mcal{H}_1(\tilde x)\hat{G}(i,j)=0$ for $i\neq j$}.
    \nonumber
\end{align}
which is quasi-convex in \((\gamma, \hat{\mcal{G}})\), allowing for an efficient solution via bisection.
\section{Sub-optimality Analysis}
\label{sec:subopt}
In this section, we prove the following sub-optimality result, which relates the performance $\hat J$ achieved by the controller synthesized via the robust synthesis problem \eqref{eqn:robust-data-sls-lqr} to the optimal performance $J^\star$ achieved by the optimal LQG controller.

To state the our main result, we let $G_k^\star$ denote the optimal solution to the $L-1-k$ (recall indexing starts at zero) horizon LQG problem, as defined in the Theorem \ref{thm:noiseless-data-sls}.  Further, let:
    \begin{align*}
       {\mcal{O}_L (A)}
            := \begin{bmatrix}
                I \\ A \\ \vdots \\ A^{L-1}
            \end{bmatrix},
        \qquad
        \mcal{T}_L (X)
            := \begin{bmatrix}
                    0 \\
                    X & 0 \\
                    AX & X & \ddots \\
                    \vdots & \vdots & \ddots & \ddots \\
                    A^{L - 2} & A^{L - 3} & \cdots & X & 0
                \end{bmatrix}.
    \end{align*}

\begin{theorem}
    \label{thm:subopt}
    Let \((\hat{\mcal{G}}, \hat{\gamma})\) be the optimal solution to
    \eqref{eqn:robust-data-sls-lqr}, let $\hat J$ be the LQG cost that the controller $\hat{\mcal K} = \hat\PHI_u\hat\PHI_x^{-1}$ constructed from the system responses \eqref{eqn:hats} achieves on system \eqref{eq:lti}.  Assume that $T\geq 2L+1$, that $\twonorm{\mcal H_L(w)}\leq \varepsilon$, and that $\varepsilon$ satisfies the bounds \eqref{eq:eps-bound}. Let $\mcal{G}^\star =
        \mrm{blkdiag}(G_0 ^\star, \ldots, G_{L - 1}^\star),$ with 
        $G_0 ^\star, \ldots, G_{L - 1} ^\star \in \Gamma(x)$ the parameters to the optimal LQG system responses as defined above, and let $\{\PHI_x^\star, \PHI_u^\star\} = \{\OZ (I_L \otimes \mcal H_L(x)) \mcal G^\star, \OZ (I_L \otimes \mcal H_L(u)) \mcal G^\star\}.$  Letting \(J^\star\) be the optimal LQG cost achieved by the resulting optimal controller $\mcal K_\star = \PHI_u^\star (\PHI_x^\star)^{-1}$, we then have that
    \begin{align*}
        \frac{\hat{J} - J^\star}{J^\star}
        \leq 6\twonorm{\mcal G^\star}\varepsilon
            \parens{2\sqrt{L} + {\norm{\mcal{T}_{T - L + 1} (I)}_2}
            +
            \frac{L (1+\twonorm{\mcal O_L(A)}) \lVert Q^{1/2} \rVert_F \twonorm{\mcal T_{T-L+1}(I)}}{J^\star}}
    \end{align*}
\end{theorem}

Our proof strategy is to construct a feasible solution to problem \eqref{eqn:robust-data-sls-lqr} using the optimal \(\mcal{G}^\star\) defined in the theorem statement, such that $\{\PHI_x^\star, \PHI_u^\star\} = \{\OZ (I_L \otimes \mcal H_L(x)) \mcal G^\star, \OZ (I_L \otimes \mcal H_L(u)) \mcal G^\star\}$ for data $\{\xt,\ut\}$ generated by system \eqref{eq:lti} with no driving noise.\footnote{Such a $\mcal G^\star$ exists by Theorem \ref{thm:noiseless-data-sls}, and we select the minimum norm $\mcal G^\star$ satisfying the desired relationship.  Future work will seek explicit relationships between the norms of the system responses, data matrices, and $\mcal G^\star$.}  First, we introduce the following technical lemma relating Hankel matrices constructed from state trajectories of system \eqref{eq:lti} with and without driving noise.

\begin{lemma}
    \label{lem:noisy-hankel-dynamics}
    Let $\xt$ and $\tilde x_{[0,T-1]}$ be the state signals for
    system \eqref{eq:lti}, driven by noise-free \(\ut\) and noisy \(\{\ut, w_{[-1,T-2]}\}\), respectively.
    Then the following holds
 \[\mcal{H}_L (\tilde x)
        =
        \mcal{H}_L (x)
        + \mcal{T}_L (I) \mcal{H}_L (w)
        + \mcal{O}_L (A) \mcal{W}_{[0, T - L]},\] 
        where 
        \(\mcal{W}(t) = \sum_{k = 0}^{t - 1} A^{t - 1 - k} w(k)\)
    are columns of  
    \(\mcal{W}_{[0, T-L]}
            = \begin{bmatrix} \mcal{W}(0) & \cdots & \mcal{W}(T - L) \end{bmatrix}\).\footnote{We let $\mathcal{W}(0) = 0$ by convention.}
\end{lemma}
\begin{proof}
    Let \(x_{[t, t-L+1]}\) and \(\tilde x_{[t, t+L-1]}\)
    be an arbitrary pair of columns of \(\mcal{H}_L (x)\) and \(\mcal{H}_L (\tilde x)\).
    Then, by the dynamics of system \eqref{eq:lti}, we have that \begin{align*}
        x_{[t, t + L - 1]} &= \mcal{O}_L (A) x(t) + \mcal{T}_L (B) u_{[t, t + L - 1]},\\
        \tilde x_{[t, t + L - 1]} &= \mcal{O}_L (A) \tilde x(t) + \mcal{T}_L (B) u_{[t, t+L-1]} + \mcal{T}_L (I) w_{[t, t+L-1]}.
    \end{align*}  
    
    Since \(\tilde x(0) = x(0)\),
    using the same control sequence \(\ut\) on both systems reveals
    the perturbation factor \(\tilde x(t) = x(t) + \mcal{W}(t)\).
    Consequently, 
    \[\tilde x_{[t, t + L - 1]}
            = \mcal{O}_L (A) (x(t) + \mcal{W}(t)) + \mcal{T}_L (B) u_{[t, t+L-1]} + \mcal{T}_L (I) w_{[t, t+L-1]}.\]  
    Extending this to all columns of the Hankel matrices completes the proof.
\end{proof}

We now use Lemma~\ref{lem:noisy-hankel-dynamics} to construct a feasible solution to the robust optimization problem \eqref{eqn:robust-data-sls-lqr} using the optimal solution $\mcal G^\star$, which is subsequently used to prove the main result of this section.

\begin{lemma}
    \label{lem:feas-soln}
    Let $\xt$ and $\tilde x_{[0,T-1]}$ be the state signals for
    system \eqref{eq:lti}, driven by noiseless \(\ut\) and noisy \(\{\ut, w_{[-1,T-2]}\}\), respectively, and suppose that \(\{\ut, w_{[-1,T-2]}\}\) are PE of order \(n + L\).  Let $\mcal{G}^\star =
    \mrm{blkdiag}(G_0 ^\star, \ldots, G_{L - 1}^\star),$ with 
    $G_0 ^\star, \ldots, G_{L - 1} ^\star \in \Gamma(x)$, be the parameter to the optimal LQG system responses in Theorem \ref{thm:noiseless-data-sls}.
    Then, if
    \begin{align}\label{eq:eps-bound}
        \varepsilon
        \leq \min \braces{
            \frac{1}{3 \sqrt{L} \norm{\mcal{G}^\star}_2},
            \enskip
            \frac{1}{2\norm{\mcal{G}^\star}_2 \cdot \norm{\mcal{T}_{T - L + 1} (I)}_2}},
    \end{align}
    the pair \(\{\mcal{G}^0 = \mcal{G}^\star (I + \mcal{D})^{-1},
        \enskip \gamma^0 = 2\varepsilon\norm{\mcal{G}^\star}_2\sqrt{ L}\}\)
    is a feasible solution to \eqref{eqn:robust-data-sls-lqr}
    where
    \begin{align*}
        \mcal{D} = \mathrm{blkdiag}(D_0, \ldots, D_{L - 1}),
        \qquad
        D_k = \mcal{W}_{[0, T-L]} G_k ^\star,
        \quad \text{for} \enskip 0 \leq k \leq L - 1.
    \end{align*}
\end{lemma}

\begin{proof}
    We first show that the candidate system responses $\mcal G^0$ satisfies the algebraic constraints of robust optimization problem \eqref{eqn:noisy-data-sls-lqr} defined in terms of the noisy data $\{\tilde x_{[0,T-1]}, \ut\}$.  First, note that since $\mcal G^0$ is block-diagonal, we trivially have that $\mcal H_1(\tilde x)G^0(i,j) = 0$ for $i\neq j$.  Thus it suffices to verify that each $G^0_k = G^\star_k (I+D_k)^{-1} \in \Gamma(\tilde x)$, i.e., that it satisfies $\mcal H_1(\tilde x)G^0_k = I$.

    First, the relation \(\mcal{H}_1 (\tilde x) = \mcal{H}_1 (x) + \mcal{W}_{[0, T-L]}\)
    follows from Lemma~\ref{lem:noisy-hankel-dynamics}
    by examining the first block row of \[
    \mcal{H}_L (\tilde x)
        =
        \mcal{H}_L (x)
        + \mcal{T}_L (I) \mcal{H}_L (w)
        + \mcal{O}_L (A) \mcal{W}_{[0, T - L]}.
    \]
    The first block row of \(\mcal{T}_L (I) \mcal{H}_L (w)\) is zero
    and the top block of \(\mcal{O}_L (A)\) is \(I\).
    With this identity, right
    multiplying \(\mcal{H}_1 (\tilde x)\) by any \(G_k ^\star \in \Gamma(x)\) yields
    \begin{align*}
        \mcal{H}_1 (\tilde x) G_k ^\star
        = \mcal{H}_1 (x) G_k ^\star + \mcal{W}_{[0, T-L]} G_k ^\star
        = I + D_k
    \end{align*}
    This implies that \(G_k^\star (I + D_k)^{-1} \in \Gamma(\tilde x)\)
    provided the inverse exists.  A sufficient condition for $(I+\mcal D)^{-1}$ to exist is that $\twonorm{\mcal D}<1$.
    We will now prove this fact, and the additional norm constraint in the robust optimization problem \eqref{eqn:noisy-data-sls-lqr}.
   
    First, the condition
    \(\gamma^0 = 2\varepsilon \norm{\mcal{G}^\star}_2\sqrt{L } \leq 2/3<1\) is satisfied
    by the assumption \eqref{eq:eps-bound}. Next,
    \begin{multline*}
        \twonorm{D_k} = \norm{\mcal{W}_{[0, T-L]} G_k ^\star}_2
        \leq \norm{\mcal{W}_{[0, T-L]}}_2 \cdot \norm{G_k ^\star}_2
        \leq \norm{\trm{vec} (W_{[0, T-L]})}_2 \cdot \norm{G_k ^\star}_2 \\
        = \norm{\mcal{T}_{T - L + 1} (I) \mcal H_{T-L+1}^\dag(\mcal H_{T-L+1}(w_{[0, T-L]}))}_2 \cdot \norm{G_k ^\star}_2 
        \leq \norm{\mcal{T}_{T - L + 1} (I)}_2 \twonorm{\mcal H_{T - L + 1} ^\dag} {\varepsilon} \norm{G_k ^\star}_2 \leq 1/2,
    \end{multline*}
    where we exploited the identity $\trm{vec} (W_{[0, T-L}) = \mcal{T}_{T - L + 1} (I) w_{[0, T-L]}$ in the second equality.  The final equalities follows from the fact that the linear map $w_{[0,T-L]} \mapsto \mcal H_{T-L+1}(w_{[0,T-L]})$ has a left-inverse $H^\dag_{T-L+1}$ with spectral norm bounded by one: this follows from the definition of the Hankel matrix, and by noting that $H^\dag_{T-L+1}$ extracts unique sub-elements of $H_{T-L+1}(w_{[0,T-L]})$ to reconstruct $w_{[0,T-L]}$, the norm bound assumption on $H_{T-L+1}(w_{[0,T-L]})$, and assumption \eqref{eq:eps-bound}.

    Thus each \(\norm{D_k}_2 \leq 1/2\), and consequently
    \(\lVert (I + \mcal{D})^{-1} \rVert_2 \leq 2\), proving that the candidate $\mcal G^0 \in \Gamma(\tilde x).$  Finally we show that \(\norm{\mcal G^0}_2 \leq \gamma^0/\varepsilon\sqrt{ L}\)
    is satisfied:
    \begin{equation*}
        \twonorm{\mcal G^0} =\twonorm{\mcal G^\star(I+\mcal D)^{-1}} \leq \frac{\twonorm{\mcal G^\star}}{1-\twonorm{\mcal D}} \leq 2\twonorm{\mcal G^\star} = \frac{2\twonorm{\mcal G^\star}\gamma^0}{2\twonorm{\mcal G^\star}\varepsilon \sqrt{L}} = \frac{\gamma^0}{\varepsilon\sqrt{ L}},
    \end{equation*}
    where the last inequality follows from the previously derived bound that $\twonorm{\mcal D}\leq 1/2$, and the final equality from the definition of $\gamma^0$.
\end{proof}

\begin{proof}[Proof of Theorem \ref{thm:subopt}]
From the derivation of \eqref{eqn:robust-data-sls-lqr},
    the cost $\hat J$ achieved by the system response matrices $\{\hat\PHI_x, \hat\PHI_u\} = \{\OZ (I_L \otimes \mcal H_L(\tilde x)) \mcal{\hat G}, \OZ (I_L \otimes \mcal H_L(u)) \mcal{\hat G}\}$ on the true dynamics is upper-bounded via
    \begin{align}
        \label{eqn:subopt-1}
        \hat{J}
            =  \norm{\begin{bmatrix} \mcal{Q}^{1/2} \\ & \mcal{R}^{1/2} \end{bmatrix}
            \begin{bmatrix} \hat{\PHI}_x \\ \hat{\PHI}_u \end{bmatrix}
            (I + \Delta)^{-1}}_F
        \leq \frac{1}{1 - \hat{\gamma}}
        \norm{\begin{bmatrix} \mcal{Q}^{1/2} \\ & \mcal{R}^{1/2} \end{bmatrix}
            \begin{bmatrix} \hat{\PHI}_x \\ \hat{\PHI}_u \end{bmatrix}}_F
    \end{align}
    where the RHS is the optimal value of \eqref{eqn:robust-data-sls-lqr}.
    Letting $\alpha(T) =  {\norm{\mcal{T}_{T - L + 1} (I)}_2\norm{\mcal G^\star}_2} $,
    then
    \begin{align*}
    \hat{J}
    &\stackrel{(a)}{\leq} \frac{1}{1 - \gamma^0}
        \norm{\begin{bmatrix} \mcal{Q}^{1/2} \\ & \mcal{R}^{1/2} \end{bmatrix}
        \begin{bmatrix}
            \OZ (I_L \otimes \mcal{H}_L (\tilde x)) \\
            \OZ (I_L \otimes \mcal{H}_L (u))
        \end{bmatrix}
        \mcal{G}_0}_F
        \\
    &\stackrel{(b)}{\leq} \frac{1}{(1 - \gamma^0)(1 - \alpha(T)\varepsilon)}
        \norm{\begin{bmatrix} \mcal{Q}^{1/2} \\ & \mcal{R}^{1/2} \end{bmatrix}
        \begin{bmatrix}
            \OZ (I_L \otimes \mcal{H}_L (\tilde x)) \\
            \OZ (I_L \otimes \mcal{H}_L (u))
        \end{bmatrix}
        \mcal{G}^\star}_F \\
    &\stackrel{(c)}{=}
        \frac{\norm{\begin{bmatrix}\mcal{Q}^{1/2} \\ & \mcal{R}^{1/2} \end{bmatrix}
                \begin{bmatrix}
                    \mcal{Z}_L (I_L \otimes \{\mcal{H}_L (x) + \mbf{\Delta}_w\} \\
                    \mcal{Z}_L (I_L \otimes \mcal{H}_L (u))
                \end{bmatrix}
                \mcal{G}^\star
            }_F}
            {(1 - \gamma^0)(1 - \alpha(T)\varepsilon)}
            \\
    &\stackrel{(d)}{\leq}
        \frac{J^\star +
            \norm{
                \begin{bmatrix}
                    \mcal{Q}^{1/2} \mbf{\Delta}_w
                    & \mcal{Z} \mcal{Q}^{1/2} \mbf{\Delta}_w
                    & \cdots
                    & \mcal{Z}^{L - 1} \mcal{Q}^{1/2} \mbf{\Delta}_w
                \end{bmatrix}}_F
            \norm{\mcal{G}^\star}_2}
        {(1 - \gamma^0) (1 - \alpha(T) \varepsilon)} \\
    &\stackrel{(e)}{\leq} \frac{J^\star +
            \sqrt{L}
            \norm{\mcal{Q}^{1/2} \mbf{\Delta}_w}_F
            \norm{\mcal{G}^\star}_2}
            {(1 - \gamma^0) (1 - \alpha(T) \varepsilon)} \\
    &\stackrel{(f)}{\leq} \frac{J^\star +
            \sqrt{L}
            \norm{\mcal{Q}^{1/2}}_F \norm{\mbf{\Delta}_w}_2
            \norm{\mcal{G}^\star}_2}
            {(1 - \gamma^0) (1 - \alpha(T) \varepsilon)} \\
    &\stackrel{(g)}{=} \frac{J^\star +
            L \norm{Q^{1/2}}_F \norm{\mbf{\Delta}_w}_2
            \norm{\mcal{G}^\star}_2}
            {(1 - \gamma^0) (1 - \alpha(T) \varepsilon)}
    \end{align*}
    
    where we use the relations
    \begin{align*}
        \mcal{G}_0 = \mcal{G}^\star (I + \mcal{D})^{-1},
        \quad
        \begin{bmatrix} \PHI_x ^\star \\ \PHI_u ^\star \end{bmatrix}
        = \begin{bmatrix}
            \OZ (I_L \otimes \mcal{H}_L (x)) \\
            \OZ (I_L \otimes \mcal{H}_L (u))
        \end{bmatrix} \mcal{G}^\star,
        \quad
        \norm{(I + \mcal{D})^{-1}}_2 \leq \frac{1}{1 - \alpha(T){\varepsilon}}
    \end{align*}
    as follows: (a) is obtained from optimality of $(\hat{\mcal G},\hat\gamma)$ and by substituting in the feasible solution $(\mcal G^0,\gamma^0)$, (b) from exploiting the compatibility of the Frobenius and spectral norm ($\|AB\|_F\leq \|A\|_2\|B\|_F$) and bounding $\|(I-\mcal D)^{-1}\|_2\leq (1-\alpha(T)\varepsilon)^{-1}$, (c) from Lemma~\ref{lem:noisy-hankel-dynamics} where we let \(\mbf{\Delta}_w := \mcal{T}_L (I) \mcal{H}_L (w) + \mcal{O}_L (A) \mcal{W}_{[0, T - L]})\), (d) from the triangle inequality, the definition of $J^\star$, the compatiblity of the Frobenius and spectral norms, and by expanding out the inner multiplication $\mcal{Z}_L (I_L \otimes \Delta_w)$, (e) by noting that $\mcal Z^k \mcal Q^{1/2} \Delta_w$ is a submatrix $\mcal Q^{1/2} \Delta_w$ and that $\|X_0, X_1, \dots, X_{L-1}\|_F = \sqrt{L}\|X_1\|_F$ when $X_i = X_0$ for all $i=0,\dots, L-1$, (f) the compatiblity of the Frobenius and spectral norms and (g) by noting that $\mcal Q^{1/2} = I_L \otimes Q^{1/2}$.
    
    Furthermore, \(\norm{\mbf{\Delta}_w}_2\)
    can be bound using Lemma~\ref{lem:feas-soln}
    and the fact that
    \(\norm{\mcal{T}_L}_2 \leq \norm{\mcal{T}_{T - L + 1}}_2\):
    \begin{align*}
    \norm{\mbf{\Delta}_w}_2
    \leq
        \norm{\mcal{T}_L (I) \mcal{H}_L (w)}_2
            +
        \norm{\mcal{O}_L (A) }_2
        \norm{\mcal{T}_{T - L + 1} (I)}_2 
        \norm{w_{[0, T - L]}}_2
 \leq (1 + \norm{\mcal{O}_L (A)}_2) \norm{\mcal{T}_{T - L + 1} (I)}_2 \varepsilon
    \end{align*}
    This further simplifies the \(\hat{J}\) bound to
    \begin{align*}
        \hat{J} \leq
            \frac{J^\star}{(1 - \gamma^0) (1 - \alpha(T) \varepsilon)}
            + \frac{L (1 + \norm{\mcal{O}_L (A)}_2) \norm{Q^{1/2}}_F
            \norm{\mcal{T}_{T - L + 1} (I)}_2
            \norm{\mcal{G}^\star}_2
            \varepsilon}
            {(1 - \gamma^0)(1 - \alpha(T)\varepsilon)}
    \end{align*}
    
    Simplifying the above and rearranging the above yields
    \begin{align*}
        \frac{J - J^\star}{J^\star}
            \leq \frac{\gamma^0 + \alpha(T){\varepsilon}}{(1 - \gamma^0) (1 - \alpha(T){\varepsilon})}
            + \frac{L (1 + \norm{\mcal{O}_L (A)}_2) \norm{Q^{1/2}}_F
            \norm{\mcal{T}_{T - L + 1} (I)}_2
            \norm{\mcal{G}^\star}_2
            \varepsilon}
            {(1 - \gamma^0)(1 - \alpha(T)\varepsilon) J^\star}
    \end{align*}
    Recall that \(\gamma^0 = 2\varepsilon\twonorm{\mcal G^\star}\sqrt{ L} \),
    and by the assumptions of the Theorem we have that \(\alpha(T){\varepsilon} \leq 1/2\) and $\gamma^0 \leq 2/3$; thus
    \begin{align*}
        \frac{\gamma^0 + \alpha(T){\varepsilon}}
                    {(1 - \gamma^0) (1 - \alpha(T){\varepsilon})}
        \leq 6 \varepsilon (2\twonorm{\mcal G^\star}\sqrt{ L} + \alpha(T)).
    \end{align*}
    Similarly, the second term can be bound as
\(6 L (1 + \norm{\mcal{O}_L (A)}_2) \norm{Q^{1/2}}_F
            \norm{\mcal{T}_{T - L + 1} (I)}_2
            \norm{\mcal{G}^\star}_2
            \varepsilon / J^\star\).
\end{proof}

\subsection{Sample Complexity}
\label{sec:complexity}
We now show \emph{trajectory averaging} can be used to mitigate the effects of noise on the data $\mcal H_L(\tilde x)$: as system \eqref{eq:lti} is LTI, given a collection of independently collected signals $\{\xtt^{(i)},\ut^{(i)},w_{[-1,T-2]}^{(i)}\}_{i=1}^N$, the averages $\bar x_{[0,T-1]} = \frac{1}{N} \sum_{i = 1}^{N} \xtt^{(i)},
    \,
    \bar u_{[0,T-1]} = \frac{1}{N} \sum_{i = 1}^{N} \ut^{(i)},
    \,
    \bar w_{[-1,T-2]} = \frac{1}{N} \sum_{i = 1}^{N} w_{[-1,T-2]}^{(i)},$
are themselves valid.  Importantly, due to the effects of averaging, we now have that the elements $\bar w(t) \iid \mathcal{N}(0,\tfrac{\sigma^2}{N}I_n).$  We then have the following result, which follows from matrix Gaussian series concentration inequalities \cite[Theorem 4.1.1]{tropp2012user}.

\begin{lemma} \label{lem:hankel-bound}
Let $\bar w_{[-1,T-2]}=\tfrac{1}{N}\sum_{i=1}^N w_{[-1,T-2]}^{(i)}$.
Then for all \(t \geq 0\),
\begin{equation}\label{eq:hankel-bound}
\mathbb P\left[\twonorm{\mcal H_L(\bar w)}\geq t \right] \leq 2nT e^{-\frac{t^2 N}{2\sigma^2 nT}}
\end{equation} 
\end{lemma}
\begin{proof}[Proof of Lemma \ref{lem:hankel-bound}]
Let $\tilde w = \frac{1}{N} \sum_{i=1}^N w^{(i)}_{[-1,2T-2]}$: then $\mcal H_{T}(\tilde w)$ is a $nT \times T$ Hankel matrix. To simplify the proof, we bound the spectral norm of $\mcal H_{T}(\tilde w)$, because as $\mcal H_L(\bar w)$ is the top left $nL \times (T-L)$ sub-matrix of $\mcal H_{T}(\tilde w)$, it follows that $\twonorm{\mcal H_L(\bar w)}\leq \twonorm{\mcal H_T(\tilde w)}$.

We write $\mcal H_T(\tilde w)=\sum_k v_k B_k$, where the $v_k\iid \mathcal N(0,1)$, and the $\{B_k\}$ are a finite sequence of fixed matrices of dimension $nT \times nT$.  In particular, letting $v_{[-1,T-2]}\ed \frac{\sqrt{N}}{\sigma}\bar w_{[-1,T-2]}$, $v_{i,j}\iid \mathcal{N}(0,1)$ denote the $j$th component of $v(j)$, and
\[
H = \begin{bmatrix}& & 1 \\ & \iddots & \\ 1 & & \end{bmatrix} \in \R ^{T \times T}, \
Z = \begin{bmatrix} & & & \\
1 & & &\\
& \ddots & & \\
& & 1 &
\end{bmatrix} \in \R ^{nT \times nT}, \ H_i = \frac{\sigma}{\sqrt{N}}H \otimes e_i,\]
where $e_i \in \R^n,\, i=1,\dots, n$ is the standard basis element.  One can then check that
\begin{equation}\label{eq:hankel-decomp}
    \mcal H_{T}(\tilde w) \ed \sum_{k=1}^n v_{0,k} H_k +\sum_{i=1}^{T-1}\sum_{k=1}^n v_{i,k} Z^i H_k.
\end{equation}
By Theorem 4.1.1 of \cite{tropp2012user}, we have that
\[
\mathbb P\left[\twonorm{\mcal H_L(\bar w)}\geq t \right] \leq 2nT e^{-\frac{t^2}{2\nu(\mcal H_{T}(\tilde w))}},
\]
where $\nu(\mcal H_{T}(\tilde w)):=\max\left\{\mathbb E \twonorm{\mcal H^\top_{T}(\tilde w)\mcal H_{T}(\tilde w)},E \twonorm{\mcal H_{T}(\tilde w)\mcal H^\top_{T}(\tilde w)}\right\}$.  Thus it suffices to show that $\nu(\mcal H_T(\tilde w)) \leq nT\sigma^2/N$ and the result follows immediately.  Using the following identities,
\[
H_k^\top Z^i(Z^\top)^iH_k = \sum_{j=1}^{nT-i}e_je_j^\top, \enskip H_k^\top (Z^\top)^iZ^iH_k = \sum_{j=i+1}^{nT}e_je_j^\top,\]
\[ (Z^\top)^iH_k H_k^\top Z^i = \left(\sum_{j=1}^{nT-i} e_je_j^\top\right )\otimes e_k e_k^\top, \enskip Z^iH_k H_k^\top  (Z^\top)^i = \left(\sum_{j=i+1}^{nT} e_je_j^\top\right )\otimes e_k e_k^\top,
\]
the decomposition \eqref{eq:hankel-decomp}, and a careful counting argument, one can check that
\[\mathbb E \twonorm{\mcal H^\top_{T}(\tilde w)\mcal H_{T}(\tilde w)} = \norm{\frac{\sigma^2}{N}TI_{nT}}_2 =\frac{\sigma^2T}{N}, \enskip E \twonorm{\mcal H_{T}(\tilde w)\mcal H^\top_{T}(\tilde w)} = \twonorm{nT I_{T}} = \frac{\sigma^2nT}{N},
\]
from which the result follows.
\end{proof}

\begin{remark}
As in \cite{alpago2020extended}, we ensure that the control input is not averaged out by simply setting $\ut^{(i)}\equiv \ut^{(1)}$ for all $i$ such that $\bar u(t) = u^{(1)}(t)$ for all $t$.
\end{remark}

Combining Lemma \ref{lem:hankel-bound} and Theorem \ref{thm:subopt}, we obtain an end-to-end sample-complexity bound.

\begin{corollary}\label{coro:complexity}
If $N \geq 2\sigma^2nT\log(2nT/\delta)\max\left\{9L\norm{\mcal{G}^\star}^2_2 ,4\norm{\mcal{G}^\star}^2_2 \norm{\mcal{T}_{T - L + 1} (I)}^2_2\right\},$
then, with probability at least $1-\delta$, we have that

\begin{align*}
    \frac{\hat J - J^\star}{J^\star}
    \leq 
    6\twonorm{\mcal G^{\star}}{\tiny\sqrt{\tfrac{2\sigma^2nT}{N}\log\left(\tfrac{2nT}{\delta}\right)}}\left(2\sqrt{ L} + {\norm{\mcal{T}_{T - L + 1} (I)}_2}\left(1+\tfrac{L(1+\twonorm{\mcal O_L(A)}) \lVert Q^{1/2} \rVert_F}{J^\star}\right)\right)
\end{align*}

\end{corollary}
\begin{proof} 
Inverting the probability bound \eqref{eq:hankel-bound}, we have that with probability at least $1-\delta$ that$\twonorm{\mcal H_L(\bar w)}\leq \sqrt{\frac{2\sigma^2nT}{N}\log\left(\frac{2nT}{\delta}\right)}.$  Thus, we have that the bounds \eqref{eq:eps-bound} are satisfied under the assumptions of the Corollary, proving the result by combining the above bound with Theorem \ref{thm:subopt}.
\end{proof}

\section{Experiments}
\label{sec:experiments}

We present experiments on the system from \cite{dean2019sample}
\begin{align*}
    A = {\begin{bmatrix} 1.01 & 0.01 & 0.00 \\ 0.01 & 1.01 & 0.01 \\ 0.00 & 0.01 & 1.01 \end{bmatrix}},
    \
    B = I, \ \sigma^2 = 0.1,\
    Q = 10^{-3} I, \ R = I,
\end{align*}
which corresponds to a slightly unstable graph Laplacian system
with input significantly more penalized than the output.

\begin{wrapfigure}{r}{.4\textwidth}
\centering
\includegraphics[width=.35\textwidth]{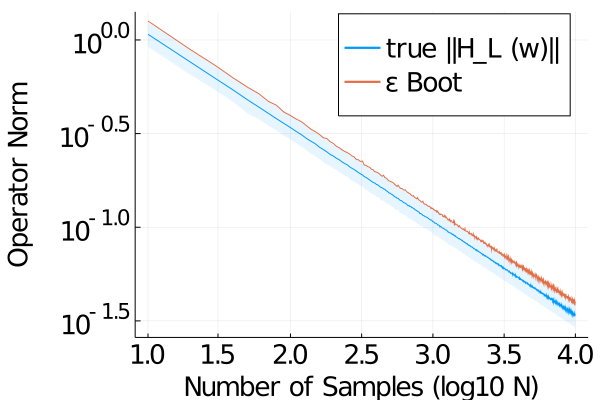}
\caption{ Bootstrap Estimation of \(\varepsilon\):
        The solid red line is the 95-th percentile bound on the bootstrapped estimate of $\twonorm{\mcal H_L(\bar w)}$ over \(1000\) trials. In blue are the median, 5-th, and 95-th percentiles of $\twonorm{\mcal H_L(\bar w)}$ computed across an additional \(1000\) independent trials.}
        \label{fig:boot}
\end{wrapfigure}

All experiments were done in Julia v1.3.1 using the JuMP v0.21.2 library with
MOSEK v9.2.9 as a backend.
Due to the structure of the resulting semidefinite programs, we found it effective to solve the dual problem using
JuMP's \code{dual\_optimizer} function.\footnote{All code is open source and available at \href{https://github.com/unstable-zeros/data-driven-sls}{https://github.com/unstable-zeros/data-driven-sls}.}

\paragraph{Bootstrap Estimation of Noise}
We use a vanilla bootstrap method \cite{chernick2011bootstrap} to empirically estimate confidence bounds on $\twonorm{\mcal H_L(\bar w)}$, for \(L = 10\) and \(T = 45\),
as a function of the number of trajectory samples $N$: Fig.~\ref{fig:boot} shows the bootstrap estimated 95-th percentile bound compared to true 95-th percentile over $1000$ trials.

\paragraph{Controller Performance in MPC Loop:}
We consider the performance of four types of unconstrained MPC controllers based on the following finite-horizon LTV feedback gains:
\(\mcal{K}^\star\) the optimal LQG controller synthesized with noise-free data;
\(\mcal{K}_B\) and \(\mcal{K}_T\) the robust controllers synthesized using the bootstrap value of
\(\varepsilon_B\) and true \(\varepsilon = \twonorm{\mcal{H}_L (w)}\) respectively in problem \eqref{eqn:robust-data-sls-lqr};
and \(\mcal{K}_N\) the naive controller is synthesized by dropping the
robustness constriant in problem~\eqref{eqn:robust-data-sls-lqr}. For the selected values of
\(N\), random trajectories \(\{\xtt^{(i)}, \ut^{(i)}, w_{[-1,T-2]}^{(i)}\}\) of length $T=45$ are generated
and used to form the appropriate Hankel matrices
\(\mcal{H}_L (\bar{\tilde{x}}), \mcal{H}_L (\bar{u}), \mcal{H}_L (\bar w)\), by averaging trajectories, where we set $u^{(1)}(t)\iid\mcal N(0,I)$, and replay $u^{(1)}$ in all subsequent trials. Each finite-horizon controller is then synthesized with the running cost matrices $(Q,R)$ specified above, and with no constraints.  In order to remove the effects of the terminal cost $Q_L$ on stability and optimality, we set \(Q_L = P_\star\), for $P_\star$ the solution to the discrete algebraic Riccati equation for the infinite horizon LQG problem specified in terms of $(A,B,Q,R)$, thus ensuring that $\mcal K^\star$ is both stabilizing and equal to the optimal infinite horizon LQG controller.  An MPC loop is then implemented over a horizon of \(H = 1000\) time-steps starting from an initial state \(x(0) = 0\) with $w(t)\iid \mcal N(0, \sigma^2 I)$ driving noise.  For robust controllers, in order to improve computational complexity, we constrain the $\hat{\mcal G}$ to be block-diagonal, with block diagonal elements $\hat G(i,i) \in \Gamma(\tilde x)$.  As off-block diagonal elements, they trivially satisfy the null-space constraints.  In doing so, we are restricting ourselves to a subset of robust achievable closed-loop responses as specified in Theorem \ref{thm:noisy-data-sls}.

We evaluate \(50\) trials at each value of \(N\),
with empirically computed control costs shown in Figs.~\ref{fig:figure}. 
We omit the nominal controllers, which are not subject to the block-diagonal constraint on the parameter matrix $\hat{\mcal G}$, as they consistently fail to stabilize the system.  While the robust controllers tend to have worse cost than the optimal controller (Fig.~\ref{fig:figure}(Left)),
they achieve better disturbance rejection, as seen by the smaller norm of the state trajectories  (Fig.~\ref{fig:figure}(Middle)), at the expense of larger control effort (Fig.~\ref{fig:figure}(Center)).

\begin{figure}
\centering
   \includegraphics[width=.245\textwidth]{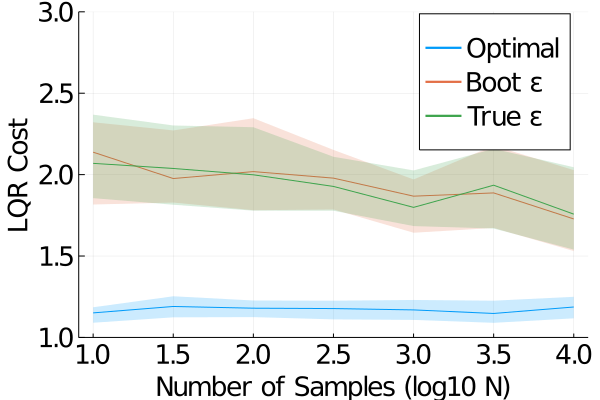}~
    \includegraphics[width=.245\textwidth]{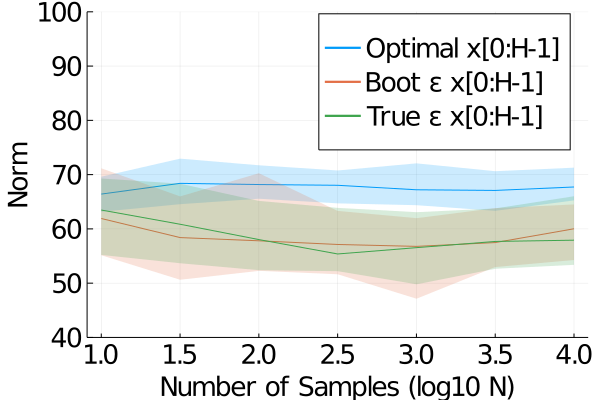}~
    \includegraphics[width=.245\textwidth]{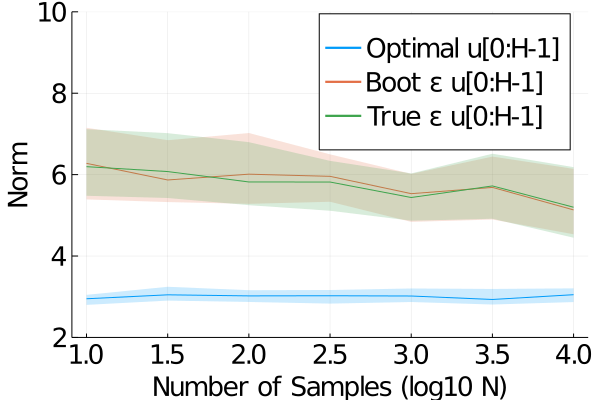}
    \caption{
    (Left) Median and quartiles of MPC controller performances,
    with infeasibility assumed to be \(+\infty\) in cost.
    (Middle, Right) Median and quartiles of running state (Middle) and input (Right) trajectory norm
    \(\twonorm{x_{[0, H-1]}}, \twonorm{u_{[0, H-1]}}\),
    with infeasibility assumed to be \(+\infty\) in norm.}
    \label{fig:figure}
\end{figure}

\section{Conclusion \& Discussion}
\label{sec:conclusion}
We defined and analyzed data-driven SLS parameterizations of stabilizing controllers for LTI systems.  We showed that when given noise free trajectories there exists an exact equivalence between traditional and data-driven SLS.  We then showed that when given noisy system trajectories, tools from robust SLS and matrix concentration theory can be used to characterize the sample-complexity needed to mitigate the effects of noise on closed-loop performance.  An important consequence of our work is a novel regularized formulation penalizing the $\ell_2\to\ell_2$, or $\mcal{H}_\infty$, norm of system responses in order to provide robustness to trajectory data driven by unmeasured noise.  We further demonstrated the importance of explicitly taking uncertainty in the trajectory data into account through experiments, where nominal data-driven methods failed.  

Future work will look to extend these results to the infinite horizon, distributed \& robust MPC, and output-feedback settings.  We believe that the bridge between behavioral and SLS methods will be particularly impactful in the distributed setting, where \emph{localized} control methods will allow for distributed controllers to be synthesized using only locally collected trajecotry data, see for example \cite{wang2018separable}.

\section{Acknowledgements}
The authors would like to thank Karl Schmeckpeper and Alexandre Amice, as well as the anonymous L4DC reviewers, for helpful feedback and comments.

\bibliographystyle{plain} 
\bibliography{sources}


\end{document}